\documentclass[11pt]{article}
\usepackage[utf8]{inputenc}
\usepackage[a4paper, margin=1in]{geometry}
\usepackage{amsmath}
\usepackage{array}
\newcolumntype{P}[1]{>{\centering\arraybackslash}p{#1}}
\newcolumntype{M}[1]{>{\centering\arraybackslash}m{#1}}
\usepackage{xcolor}
\usepackage{amsfonts, amssymb}
\usepackage{amsthm}
\usepackage{esvect}
\usepackage[english]{babel}
\usepackage[ruled,linesnumbered]{algorithm2e}
\usepackage[T1]{fontenc}
\usepackage{multirow}
\usepackage{graphicx}
\usepackage{cite}
\usepackage[numbers,sort]{natbib}
 \usepackage{titlesec}
 \usepackage{caption}
\titleformat{\section}
  {\normalfont\fontfamily{ptm}\fontsize{11}{11}\bfseries}{\thesection}{1em}{}
\titleformat{\subsection}
  {\normalfont\fontfamily{ptm}\fontsize{10}{11}\bfseries}{\thesubsection}{1em}{}
\titleformat{\subsubsection}
  {\normalfont\fontfamily{ptm}\fontsize{10}{11}\selectfont}{\thesubsubsection}{1em}{}
\newtheorem{theorem}{Theorem}[section]
\newtheorem{corollary}[theorem]{Corollary}
\newtheorem{lemma}[theorem]{Lemma}
\newtheorem{proposition}[theorem]{Proposition}
\newtheorem{definition}[theorem]{Definition}
\newtheorem{remark}[theorem]{Remark}
\usepackage{hyperref}
\hypersetup{
    colorlinks=true,    
    urlcolor=cyan,
    linkcolor=cyan,
    citecolor=cyan
}
\usepackage{mathtools}

\DeclarePairedDelimiter\floor{\lfloor}{\rfloor}
\providecommand{\keywords}[1]
{
  \small	
  \quad \quad \textbf{\textit{Keywords --}} #1
}

\newcommand{\bfs}[1]{\mbox{\boldmath$#1$}}

\title{\large An Exact Cutting Plane Method for $k$-submodular Function Maximization}

\author{ \small Qimeng Yu \quad Simge K\"u\c{c}\"ukyavuz\vspace{0.2cm}  \\ \small Department of Industrial Engineering and Management Sciences \\ \small Northwestern University, Evanston, IL, USA \\ \small \{kim.yu@u.northwestern.edu, simge@northwestern.edu\}}
\date{\small \today} 
\begin{document}
\maketitle

\begin{abstract}
\noindent  A natural and important generalization of submodularity---$k$-submodularity---applies to set functions with $k$ arguments and appears in a broad range of applications, such as infrastructure design, machine learning, and healthcare. In this paper, we study maximization problems with $k$-submodular objective functions. We propose valid linear inequalities, namely the $k$-submodular inequalities, for the hypograph of any $k$-submodular function. This class of inequalities serves as a novel generalization of the well-known submodular inequalities. We show that maximizing a $k$-submodular function is equivalent to solving a mixed-integer linear program with exponentially many $k$-submodular inequalities. Using this representation in a delayed constraint generation framework, we design the first exact algorithm, that is not a complete enumeration method, to solve general $k$-submodular maximization problems. Our computational experiments on the {multi-type sensor placement problems} demonstrate the {efficiency} of our algorithm in constrained nonlinear $k$-submodular maximization problems for which no alternative compact mixed-integer linear formulations are available. The computational experiments show that our algorithm significantly outperforms the only available  exact solution method---exhaustive search. Problems that would require over 13 years to solve by exhaustive search can be solved within ten minutes using our method. 
\end{abstract}

\section{Introduction}
\label{sect:intro}
Submodularity is an important concept in integer and combinatorial optimization. Several functions of great theoretical interest in combinatorial optimization are submodular, such as the set covering function and the graph cut function. Submodularity also arises in numerous practical applications, including the influence maximization problem \citep{kempe2015maximizing}, the sensor placement problem \citep{krause2008efficient}, and the hub location problem \citep{contreras2014hub}. Submodularity is shown to be the discrete analogue of convexity \citep{lovasz1983submodular}, and the unconstrained submodular \emph{minimization} problem is polynomially solvable \citep{iwata2001combinatorial,schrijver2000combinatorial,orlin2009faster, lee2015faster}. However, submodular minimization with simple constraints, such as cardinality constraints, is generally NP-hard \citep{svitkina2011submodular}. Submodular \emph{maximization} is known to be NP-hard even in the unconstrained case.  \citet{nemhauser1978analysis} prove that the greedy method for maximizing a monotone submodular function subject to a cardinality constraint is a $(1-1/e)$-approximation algorithm.\\

In addition to the approximation algorithms, polyhedral approaches have been popular in submodular optimization research. \citet{edmonds1970submodular} gives linear inequalities which fully describe the convex hull of the epigraph for a submodular function (see also \cite{edmonds2003submodular}). These inequalities are referred to as extended polymatroid inequalities, as they are closely related to a structure called extended polymatroid. \citet{atamturk2008polymatroids} establish a polarity result analogous to the relationship between extended polymatroids and extended polymatroid inequalities for non-submodular functions. With this observation, \citet{atamturk2021submodular} present an alternative proof for the convex hull description of the epigraphs for submodular functions. \citet{nemhauser1988integer}  give an exact method for submodular maximization from a polyhedral perspective. They show that maximizing any submodular function is equivalent to solving a mixed-integer linear program with exponentially many linear inequalities, referred to as the submodular inequalities. \citet{atamturk2021submodular} provide valid inequalities for general set functions by exploiting their submodular-supermodular decomposition. Moreover, the polyhedral approach has received renewed interest, both in terms of strengthening extended polymatroid inequalities \citep{yu2017polyhedral,yu2021strong} and submodular inequalities \citep{Ahmed2011,Yu2017,shi2020sequence}, as well as extending their use to stochastic settings \citep{Wu2017,Wu2018maxinf,Wu2020, zhang2018ambiguous, xie2019distributionally}. In addition, submodular properties in mixed-binary convex quadratic and conic optimization problems are discovered and exploited in \citep{atamturk2008polymatroids,gomez2018submodularity, atamturk2020submodularity, atamturk2020supermodularity, Kilinc-Karzan2020-conicbinary}. \\

Submodularity can be generalized to functions with $k$ set arguments for any positive integer $k$, resulting in the concept called \emph{$k$-submodularity}. This term is introduced by \citet{huber2012towards}, and it encompasses submodularity and \emph{bisubmodularity} \citep{chandrasekaran1988pseudomatroids,qi1988directed} as special cases when $k=1$ and $k=2$, respectively.  Researchers have studied \emph{$k$-submodular minimization}, where $k\geq 2$, using various approaches. \citet{qi1988directed} proves an analogue of Lov\'asz extension for bisubmodular functions, implying that this class of functions can be minimized in polynomial time using the ellipsoid method. Subsequently, weakly and strongly polynomial-time algorithms are proposed for unconstrained bisubmodular minimization \citep{fujishige2005bisubmodular,mccormick2010strongly}. \citet{yu2020polyhedral} take a polyhedral approach and present a complete linear convex hull description for the epigraph of any bisubmodular function. Based on this polyhedral characterization, the authors propose an effective cutting plane algorithm to solve constrained bisubmodular minimization. \citet{huber2012towards} generalizes the Min-Max Theorem for submodular and bisubmodular minimization to the $k$-submodular case with $k\geq 3$. Whether $k$-submodular functions can be minimized in polynomial-time when $k\geq 3$ is still an open problem.\\

The \emph{$k$-submodular maximization} problem---a generalization of the NP-hard submodular maximization problem---is also NP-hard. Extensive research has been devoted to developing approximation algorithms and proving their guarantees. For example, \citet{singh2012bisubmodular} give a constant-factor approximation algorithm for a class of bisubmodular functions. {The authors refer to bisubmodularity that we consider in this paper as directed bisubmodularity. They show that a bisubmodular function can be embedded into a submodular function over an extended ground set,  a set containing two copies of each element in the original ground set. For each subset of the extended ground set that contains both copies of an element, the submodular function value can be recursively obtained by solving discrete optimization problems with exponentially sized decision spaces. Such construction is computationally expensive but theoretically interesting.} \citet{iwata2013bisubmodular, ward2014maximizing} independently show that a randomized greedy algorithm attains the approximation guarantee of $1/2$ for unconstrained bisubmodular maximization (see also \citep{ward2016maximizing}). For $k$-submodular maximization with $k\geq 3$, \citet{ward2016maximizing} achieve a $\max(1/3, 1/(1+a))$-approximation, where $a=\max(1, \sqrt{(k-1)/4})$. \citet{iwata2016improved} improve this result to a factor of $1/2$ and show that there is a $(k/(2k-1))$-approximation algorithm for unconstrained monotone $k$-submodular maximization. {The authors further prove that their algorithms are asymptotically tight. In terms of  \emph{constrained} $k$-submodular maximization, researchers predominantly focus on non-negative monotone $k$-submodular functions. For example, \citet{ohsaka2015monotone} propose greedy algorithms for maximizing non-negative monotone $k$-submodular functions subject to a total cardinality constraint on all selected items, and to individual cardinality constraints on each of the $k$ subsets. Their algorithm achieves a 1/2 approximation ratio under a total cardinality constraint, and this ratio is asymptotically tight. For the latter, the proposed algorithm achieves 1/3-approximation. \citet{sakaue2017maximizing} studies maximization of non-negative monotone $k$-submodular functions under matroid constraints. The proposed greedy algorithm yields a 1/2-approximate solution.} \\

{To the best of our knowledge, there is no algorithm for maximizing possibly non-monotone $k$-submodular functions under general constraints.} There has also been a paucity of research on \emph{exact} solution methods for $k$-submodular maximization, {both constrained and unconstrained. Therefore, we are interested in developing a versatile exact solution method.} {Practically, in situations where we have limited computing resources and an approximate solution suffices, approximation algorithms are extremely valuable. On the other hand, for problems such as high capital investments and strategic decision-making, where optimality is important and more computing resources are available, it is desirable to apply an exact method that solves the problem within a reasonable amount of time.} Given the formidable computational burden of exhaustive search, there are no known efficient exact methods for constrained $k$-submodular maximization. To bridge this gap, our paper takes a polyhedral approach and proposes the first computationally attractive exact solution method {that handles non-monotone $k$-submodular functions, as well as arbitrary linear constraints}. \\

Before we summarize our results, we provide a few examples from a wide range of applications of $k$-submodular maximization. 
\subsection{Multi-type sensor placement} 
\label{subsect:sensor}
Sensor networks---enabled by internet of things (IoT) technology---provide real-time monitoring and control of systems to operate smart cities \citep{zanella2014internet}, smart homes \citep{ghayvat2015wellness}, and smart grids \citep{abujubbeh2019software}. These applications often call for multiple types of sensors in the network. For example, in smart water distribution networks, multiple types of sensors are placed to measure different aspects of water quality in real time \citep{Singapore2016ManagingTW}. Suppose we have $k$ types of sensors and a set $N$ of $n$ locations to place them. If at most one sensor is allowed in each location, then every $k$-tuple of pairwise disjoint subsets of $N$ corresponds to a multi-type sensor placement plan. The effectiveness of a sensor deployment plan can usually be evaluated using $k$-submodular functions such as the entropy function. Thus a multi-type sensor placement problem can be expressed in terms of constrained $k$-submodular maximization.  \\

We provide more details of \emph{coupled sensor placement}, in which we have two types of sensors for different measurements, such as temperature and humidity. {Here, $k=2$ is an arbitrary choice for illustration purposes. The description below can be generalized to the cases with $k\geq 3$. } Every biset $(S_1, S_2)\in 3^N$ corresponds to a coupled sensor placement plan, in which the type-1 sensors are placed at the locations in $S_1$ and the type-2 sensors are placed at  $S_2$. Due to a limited budget, we can place at most $B_1$ type-1 and $B_2$ type-2 sensors. We evaluate each sensor deployment plan using entropy, which measures how much uncertainty in the environment the sensors can capture \citep{ohsaka2015monotone}. The entropy of a discrete random variable $X$ with support $\mathcal{X}$ is computed by 
\[H(X) = -\sum_{x\in \mathcal{X}} \mathbb{P}(X = x)\log \mathbb{P}(X = x). \]
The entropy of $X$ is high if multiple outcomes occur with similar probabilities, making it difficult for us to predict what we may observe. For instance, it is harder for us to guess the outcome of throwing a fair dice correctly than that of a biased dice, so the entropy in the case of a fair dice is higher than the latter. In the context of coupled sensor placement, a discrete random variable $X_{S_1, S_2}$ captures the possible observations reported by sensors installed at $(S_1, S_2)\in 3^N$, and the set $\mathcal{X}_{S_1,S_2}$ contains all possible observations. The entropy of $X_{S_1,S_2}$ is 
\[H(X_{S_1,S_2}) = -\sum_{x\in \mathcal{X}_{S_1,S_2}} \mathbb{P}(X_{S_1,S_2} = x)\log \mathbb{P}(X_{S_1,S_2} = x). \] In an ideal coupled sensor placement plan, sensors are installed at locations where the corresponding  observations are the most unpredictable. In other words, a placement $(S_1^*, S_2^*)$ is the best when  $H(X_{S_1^*,S_2^*})$ is maximal among all feasible $(S_1, S_2)\in 3^N$. \\

{As mentioned earlier, the description above holds for $k\geq 2$ types of sensors. The function $f:(k+1)^N \rightarrow \mathbb{R}$, defined by $f(S_1, \dots, S_k) = H(X_{S_1,\dots, S_k})$ for all $(S_1, \dots, S_k)\in (k+1)^N$, is monotone and $k$-submodular \citep{ohsaka2015monotone}. Thus the multi-type sensor placement problem is a cardinality-constrained $k$-submodular maximization problem with objective function $f$. }

{\subsection{Multi-topic influence maximization}
Social networks have allowed information to be spread faster than ever. In applications such as viral marketing, we may find a class of $k$-submodular maximization problems in which we aim to maximize the spread of information on $k$ topics over a social network. Suppose one would like to promote $k$-types of products. He or she may select influential network users to share their experiences of the products with their followers. As these followers share again with their own followers, messages about the products gradually diffuse across the network and reach a possibly large population. \citet{kempe2015maximizing} propose a diffusion model called independent cascade to describe the diffusion mechanism of a single type of influence. The authors also show that the expected total number of reached network users is a submodular function of the initial source of the spread. \citet{ohsaka2015monotone} generalize this model to allow $k\geq 2$ types of influence. In their model, a social network is represented by a digraph $G=(N,A)$, in which $N=\{1,2,\dots, n\}$ is the set of network users, and the arcs $A$ capture how the users interact with each other. Every arc $(i,j)\in A$ is associated with probabilities $p_{(i,j)}^q$ for every $q\in\{1,2,\dots,k\}$. Each probability $p_{(i,j)}^q$ is the likelihood of $j$ accepting $i$'s information on the $q$-th topic. Once $j$ adopts a new piece of information, this user is ready to influence his or her own neighbors. Let $S_q\subseteq N$ be the group of influencers responsible for promoting the $q$-th type of products and $A_q(S_q)$ be the individuals influenced by the initial spreaders $S_q$ about product $q$ under the stochastic model described above. Each influencer is restricted to accepting at most one  type of sponsored product for fairness. Thus the initial influencers form a $k$-tuple of pairwise disjoint subsets of $N$. The function $f:(k+1)^N\rightarrow \mathbb{R}$ defined by $f(S_1, \dots, S_k) = \mathbb{E}|\bigcup_{q=1}^k A_q(S_q)|$ computes the expected total number of influenced individuals given $k$ sets of initial influencers. This function is shown to be monotone and $k$-submodular \citep{ohsaka2015monotone}.}

\subsection{Multi-class feature selection} 
Feature selection plays a key role in multiple fields of research including machine learning \citep{shalev2014understanding}, bioinformatics \citep{saeys2007review}, and data mining \citep{rokach2008data}. This process improves the analysis of large datasets by reducing the dimensionality of data. In the resulting multi-class feature selection  problems, there are $k$ uncorrelated prediction variables, with their associated features mixed in a pool. The task is not only to find the most informative features, but also to classify the features with respect to the prediction variables, giving rise to a $k$-submodular optimization problem. \\

With two prediction variables, such problem is referred to as \emph{coupled feature selection} in \citep{singh2012bisubmodular}. More formally, suppose that a Gaussian graphical model and a set of features $N$ are given. Let $C_1, C_2$ be two variables to be predicted. The goal is to partition the features in $N$ into two sets $S_1$ and $S_2$, such that $S_1$ is used to predict $C_1$, and $S_2$ is used to predict $C_2$. It is assumed that $S_1, S_2$ are mutually conditionally independent given $C = \{C_1,C_2\}$. The total number of selected features, $|S_1|+|S_2|$, is no more than a given number $k$. Next, we describe the score function that evaluates the mutual information obtained by a coupled feature selection. 
Suppose $X$ and $Y$ are discrete random variables, and $\mathcal{X},\mathcal{Y}$ are the respective supports. Then the conditional entropy of $X$ given $Y$ is 
\[H(X|Y) = -\sum_{x\in \mathcal{X},y\in\mathcal{Y}} \mathbb{P}(X=x, Y=y) \log \frac{\mathbb{P}(X=x, Y=y)}{\mathbb{P}(Y = y)}. \] 
The biset mutual information is computed by 
\[I(S_1,S_2; C) = H(S_1\cup S_2) - \sum_{i\in S_1}H(i\mid C_1) - \sum_{j\in S_2}H(j\mid C_2), \] 
where $H$ is the entropy function discussed in Section \ref{subsect:sensor} and the conditional entropy is defined above. Intuitively, the features with higher mutual information are more informative about both prediction tasks. Let $f(S_1,S_2) = I (S_1,S_2;C)$. The function $f$ is monotone and bisubmodular \citep{singh2012bisubmodular}, and the best features can be found by maximizing $f$.

\subsection{Drug-drug interaction detection}
Drug-drug interactions (DDIs) detection is an important application in the healthcare domain which exploits bisubmodularity. DDIs are the reactions resulting from using multiple drugs concomitantly. DDIs are a major cause of morbidity and mortality \citep{lu2015novel}---adverse drug events cause 770,000 injuries and deaths every year, and as much as 30\% of these adverse drug events are due to DDIs \citep{tatonetti2012novel, pirmohamed1998drug}. \citet{hu2019bi} show that the correlations among the combinations of drugs and associated symptoms can be captured by a bisubmodular function, and the potential DDIs are determined by maximizing this function. \\

The aforementioned applications are solved using approximation algorithms due to the lack of exact solution methods. 

\subsection{Our contributions}
{Despite the developments in approximation algorithms for unconstrained and a few classes of constrained $k$-submodular maximization, there is no known exact method other than exhaustive search for general $k$-submodular maximization. To bridge this gap, we propose the first polyhedral approach to study $k$-submodular function maximization and provide an exact algorithm to maximize any $k$-submodular function subject to general constraints.} We propose a new class of valid linear inequalities called $k$-submodular inequalities. {These inequalities are non-trivial extensions of the submodular inequalities introduced by \citet{nemhauser1988integer}, in that the proposed $k$-submodular inequalities account for the interchanges of elements among the $k$ subsets.} {With these valid inequalities, we develop an exact cutting-plane algorithm for constrained $k$-submodular maximization problems, which does not require the $k$-submodular objective function to satisfy any restrictive assumptions, such as monotonicity and non-negativity, nor does it restrict the structure of the constraints.} We demonstrate the effectiveness of our algorithm by experimenting on the {multi-type  sensor placement problem, which has a highly nonlinear $k$-submodular objective function.} The computational experiments show that our algorithm significantly outperforms the exhaustive search method.

\subsection{Outline}
The outline of this paper is as follows. In Section \ref{sect:prelim}, we provide formal definitions of $k$-submodularity and review its known properties. In Section \ref{sect:properties}, we state and prove additional properties of $k$-submodular functions that have not been studied in the literature. These properties are used to establish our main results. Next,  in Section \ref{sect:k_sub_ineq}, we propose a class of valid linear inequalities which we call the \emph{$k$-submodular inequalities} for the hypograph of any $k$-submodular function. In particular, we show that maximizing a $k$-submodular function is equivalent to solving a mixed-integer linear program with exponentially many $k$-submodular inequalities. In Section \ref{sect:cut_plane}, we give a cutting plane algorithm to solve the constrained maximization problems with $k$-submodular objective functions. We demonstrate the {efficiency} of our proposed algorithm and compare it against the exhaustive search method on the {multi-type sensor placement problem} in Section \ref{sect:comp}. Lastly, we conclude with a few remarks in Section \ref{sect:conclusion}.

\section{Preliminaries}
\label{sect:prelim}
Let $N=\{1,2, \dots, n\}$ be a non-empty finite set. For any integer $k\geq 1$, we let 
\[(k+1)^N = \{(S_1,S_2, \dots, S_k): S_q\subseteq N, S_q\cap S_{q'} = \emptyset, \text{ for all } q, q' \in \{1,2,\dots,k\} \text{ with } q \neq q'\}\] 
be the collection of all $k$-sets, which are $k$-tuples of pairwise disjoint subsets of $N$. For brevity, we denote any $(S_1,S_2, \dots, S_k)\in (k+1)^N$ by $\mathbf{S}$. We call $\mathbf{S}\in (k+1)^N$ a \emph{partition}, or an \emph{orthant}, of $N$, if $\bigcup_{q=1}^k S_q = N$. 
  
\begin{definition}
\label{def:k_sub}
For any integer $k\geq 1$, a function $f:(k+1)^N \rightarrow \mathbb{R}$ is \emph{$k$-submodular} if for any $\mathbf{X}=(X_1, X_2, \dots, X_k), \mathbf{Y}=(Y_1, Y_2, \dots, Y_k)\in (k+1)^N$, 
\begin{align*}
f(\mathbf{X}) + f(\mathbf{Y}) \geq &\hspace{0.2cm} f(\mathbf{X} \sqcap \mathbf{Y}) + f(\mathbf{X} \sqcup \mathbf{Y}),
\end{align*}
where \[\mathbf{X} \sqcap \mathbf{Y} = (X_1\cap Y_1, X_2 \cap Y_2, \dots, X_k \cap Y_k),\] and \[\mathbf{X} \sqcup \mathbf{Y} = \left((X_1\cup Y_1 )\backslash \bigcup_{q=2}^k(X_q\cup Y_q), \dots, (X_k \cup Y_k)\backslash \bigcup_{q=1}^{k-1} (X_q\cup Y_q)\right).\] 
\end{definition}

In particular, the functions satisfying Definition \ref{def:k_sub} when $k=1$ are called \emph{submodular functions}, and when $k=2$ such functions are referred to as \emph{bisubmodular functions}. In the following discussion, we assume $k\geq 2$ unless specified otherwise. For any $q \in \{1,2,\dots, k\}$, $i\in N\backslash \bigcup_{q'\in\{1,\dots,k\}\backslash \{q\}} X_{q'}$ and $\mathbf{X} \in (k+1)^N$, we define 
\[\rho_{q, i}(\mathbf{X}) = f(X_1, \dots, X_q \cup\{i\}, \dots, X_k) - f(\mathbf{X}). \]
Intuitively, $\rho_{q, i}(\mathbf{X})$ represents the marginal contribution of adding $i\in N$ to the $q$-th subset of $\mathbf{X}$. \citet{ando1996characterization} provide an alternative definition of bisubmodularity that involves the notion of marginal contribution. \citet{ward2016maximizing} generalize this result to $k$-submodularity. Before explaining this equivalent definition of $k$-submodular functions, we first establish a new term. 

\begin{definition}
A function $f:(k+1)^N \rightarrow \mathbb{R}$ is \emph{submodular over a partition} $\mathbf{S} = (S_1, S_2, \dots, S_k)$ if the function
\begin{equation}
\label{eq:k_sub_partition}
\hat f_{\mathbf{S}}(X) := f(X\cap S_1, X\cap S_2, \dots, X\cap S_k)
\end{equation}
is submodular over $X\subseteq N$. 
\end{definition}

\begin{lemma} \cite[]{ward2016maximizing}
For an integer $k\geq 2$, a function $f:(k+1)^N \rightarrow \mathbb{R}$ is $k$-submodular if and only if 
\begin{itemize}
\item[{(C1)}] the function $f$ is submodular over every partition of $N$, and 
\item[{(C2)}] given any $\mathbf{X}\in (k+1)^N$ and any $i\in N\backslash \bigcup_{p=1}^k X_p$, $\rho_{q, i}(\mathbf{X}) + \rho_{q', i}(\mathbf{X}) \geq 0$ for every pair of $q, q'\in\{1,2,\dots, k\}$ such that $q \neq q'$. 
\end{itemize}
\end{lemma}
Although \citet{ward2016maximizing} assume the $k$-set functions to be non-negative, shifting such functions by a constant does not affect their $k$-submodularity. The next corollary immediately follows from condition (C1). It captures the diminishing marginal return property of $k$-submodular functions over every partition. 
\begin{corollary}
\label{lemma:marginal}
If $f$ is a $k$-submodular function, then for any $\mathbf{X},\mathbf{Y}\in (k+1)^N$ that satisfy $X_p \subseteq Y_p$ for all $p \in \{1,\dots, k\}$, $\rho_{q, i}(\mathbf{X}) \geq \rho_{q, i}(\mathbf{Y})$ for all $i\in N\backslash \bigcup_{p=1}^k Y_p$ and $q\in\{1,\dots, k\}$. 
\end{corollary}

\begin{definition} 
A $k$-submodular function $f$ over a ground set $N$ is \emph{monotone non-decreasing} if for any $\mathbf{X},\mathbf{Y}\in (k+1)^N$ such that $X_q \subseteq Y_q$ for all $q \in \{1,\dots, k\}$, the property $f(\mathbf{Y})\geq f(\mathbf{X})$ holds.
\end{definition} Equivalently, $f$ is monotone non-decreasing if for any $\mathbf{X} \in (k+1)^N$ and $i\in N\backslash \bigcup_{p=1}^k X_p$, $\rho_{q,i}(\mathbf{X}) \geq 0$ for all $q \in \{1,\dots, k\}$. We call a monotone non-decreasing function simply a \emph{monotone} function.\\

Without loss of generality, we assume that $f(\bfs{\emptyset}) = 0$ where $\bfs{\emptyset}$ is the $k$-set $(\emptyset, \dots, \emptyset)$. By slightly abusing notation, we let $f(\mathbf{X}) = f(\mathbf{x})$, where $\mathbf{x} = {[{x^1}, \dots, {x^k}]}^\top$ and $x_q^\top \in \{0,1\}^{n} = \mathbb{B}^{n}$ for every $q\in\{1,\dots,k\}$. To be more precise, $x^q_i = 1$ if $i\in X_q$, and $x^q_i = 0$ otherwise for $i\in N$ and $q\in\{1,\dots, k\}$. This is a unique one-to-one mapping between $(k+1)^N$ and $\{\mathbf{x}\in\mathbb{B}^{kn}: \sum_{q = 1}^k x^q_i \leq 1 \text{ for all } i\in N \}$. The hypograph of $f$ is
\[\mathcal{T}_f = \left\{(\mathbf{x},\eta)\in \mathbb{B}^{kn} \times \mathbb{R} : \eta \leq f(\mathbf{x}), \sum_{q = 1}^k x^q_i \leq 1 \text{ for all } i\in N \right\}. \]
 
In this study, we consider maximization problems with $k$-submodular objective functions, namely
\begin{equation}
\label{eq:original_max}
\max_{\mathbf{X} \in \mathcal{X}} f(\mathbf{X}),
\end{equation} where $f$ is $k$-submodular and $\mathcal{X}\subseteq (k+1)^N$ denotes the collection of feasible $k$-sets. When the problem is unconstrained, $\mathcal{X}$ is $(k+1)^N$. Let $\mathcal{K}$ be the set of incidence vectors $\mathbf{x}$ that correspond to the feasible $k$-sets in $\mathcal{X}$. Problem (\ref{eq:original_max}) can be rewritten as 
\begin{equation}
\label{eq:hypograph}
\max\{\eta : (\mathbf{x}, \eta)\in \mathcal{T}_f, \mathbf{x}\in \mathcal{K}\}. 
\end{equation} 
 
In Section \ref{sect:k_sub_ineq}, we propose a set of valid linear inequalities for $\mathcal{T}_f$. By using these inequalities in a cutting plane framework, we propose the first computationally feasible exact method to solve problem \eqref{eq:hypograph} in Section \ref{sect:cut_plane}.  Before we do so, we first identify additional properties of $k$-submodular functions in the next section.

\section{New Properties of $k$-submodular Functions}
\label{sect:properties}
In this section, we establish a few properties of $k$-submodular functions that are not previously discussed in the literature to the best of our knowledge. These properties are useful for deriving valid linear inequalities for $\mathcal{T}_f$ in Section \ref{sect:k_sub_ineq}. 

\begin{lemma}
\label{lemma:k_sub_monotone}
Given a ground set $N = \{1,2,\dots, n\}$, a function $f:(k+1)^N \rightarrow \mathbb{R}$ is $k$-submodular and monotone if and only if 
\begin{equation}
\label{eq:k_sub_monotone}
\hat f_{\mathbf{S}}(Y) \leq \hat f_{\mathbf{S}}(X) + \sum_{i\in Y\backslash X} [\hat f_{\mathbf{S}}(X\cup\{i\})-\hat f_{\mathbf{S}}(X)]
\end{equation} for any $X,Y\subseteq N$ over any partition $\mathbf{S}$ of $N$.
\end{lemma}
\begin{proof}
\citet{nemhauser1988integer} show that a set function $g:2^N\rightarrow \mathbb{R}$ is submodular and monotone if and only if \[ g(T) \leq g(S) + \sum_{j\in T\backslash S} [g(S\cup \{j\}) - g(S)] \text{ for any } S,T\subseteq N.\] Suppose $f$ is $k$-submodular and monotone. Given any partition $\mathbf{S}$, $\hat f_{\mathbf{S}}$ is submodular by (C1). For any $P\subseteq Q \subseteq N$, we construct $\mathbf{P},\mathbf{Q}$ such that $P_q = P\cap S_q$ and  $Q_q = Q\cap S_q$ for all $q\in \{1,\dots, k\}$. Since $P_q \subseteq Q_q$ for all $q$ and $f$ is monotone, $\hat f_{\mathbf{S}}(P) = f(\mathbf{P}) \leq f(\mathbf{Q}) = \hat f_{\mathbf{S}}(Q)$, which implies that $\hat f_{\mathbf{S}}$ is monotone. Thus property \eqref{eq:k_sub_monotone} holds. Conversely, suppose \eqref{eq:k_sub_monotone} is true. Then $\hat f_{\mathbf{S}}$ is submodular and monotone over any partition $\mathbf{S}$ of $N$, and (C1) immediately follows. Let any $\mathbf{X}\in (k+1)^N$ and $i\in N\backslash \bigcup_{p=1}^k X_p$ be given. For any $q\in\{1,\dots, k\}$, we construct a partition $\mathbf{S}^q$ such that $S_p^q = X_p$ for all $p\in\{1,\dots, k\}\backslash \{q\}$, and $S_q^q = N\backslash \bigcup_{p\in\{1,\dots, k\}\backslash \{q\}} X_p$. We note that $\mathbf{X}$ and $(X_1,\dots,X_q\cup\{i\}, \dots, X_k)$ are both in the partition $\mathbf{S}^q$. Now $\rho_{q, i}(\mathbf{X}) = \hat f_{\mathbf{S}^q}(X\cup\{i\}) - \hat f_{\mathbf{S}^q}(X) \geq 0$. It follows that $\rho_{q, i}(\mathbf{X}) + \rho_{q', i}(\mathbf{X})\geq 0$ for any $q, q'\in \{1,\dots, k\}$ with $q\neq q'$. Therefore, $f$ is monotone and (C2) holds. We conclude that $f$ is $k$-submodular and monotone. 
\end{proof}

Given a non-monotone submodular function $g$ defined over a ground set $N$, \citet{nemhauser1988integer} show that $g^*(S) := g(S) - \sum_{i\in S} (f(N) - f(N\backslash \{i\}))$ is monotone and submodular. Lemma \ref{lemma:convert_to_monotone} generalizes this result to non-monotone $k$-submodular functions.
\begin{lemma}
\label{lemma:convert_to_monotone}
Let $f:(k+1)^N \rightarrow \mathbb{R}$ be a $k$-submodular function. For every $i\in N$ and $q\in\{1,\dots,k\}$, we define 
\begin{equation*}
\label{eq:xi}
\xi_i^q = \min\left\{\rho_{q, i}(\mathbf{S}) : \mathbf{S}\in (k+1)^{N\backslash\{i\}}, \bigcup_{p=1}^k S_p = N\backslash\{i\} \right\}.
\end{equation*}
The function 
\begin{align*}
f^*(\mathbf{X}) := f(\mathbf{X}) - \sum_{q=1}^k \sum_{i\in X_q} \xi_i^q
\end{align*}
is $k$-submodular and monotone. 
\end{lemma}
\begin{proof}
By Lemma \ref{lemma:k_sub_monotone}, it suffices to show that $\hat f^*_{\mathbf{T}}$ is submodular and monotone for any partition $\mathbf{T}$. Consider any $X,Y\subseteq N$ and any partition $\mathbf{T}$ of $N$. Let $\mathbf{X}$ and $\mathbf{Y}$ be the corresponding $k$-sets over this partition. In other words, $X_q = X\cap T_q$ and $Y_q = Y\cap T_q$ for all $q\in \{1,\dots, k\}$. Then 
\begingroup
\allowdisplaybreaks
\begin{subequations}
\begin{align}
  &   \hat f^*_{\mathbf{T}}(X) + \sum_{i\in Y\backslash X} [\hat f^*_{\mathbf{T}}(X\cup\{i\})-\hat f^*_{\mathbf{T}}(X)] \\
  = \hspace{3pt}   & f^*(\mathbf{X}) + \sum_{q=1}^k \sum_{i\in Y_q\backslash X_q} [f^*(X_1, \dots, X_q\cup\{i\}, \dots, X_k) - f^*(\mathbf{X})] \label{eq:rewrite1}\\
  = \hspace{3pt}   & f(\mathbf{X}) - \sum_{q=1}^k \sum_{i\in X_q} \xi_i^q + \sum_{q=1}^k \sum_{i\in Y_q\backslash X_q} [\rho_{q,i}(\mathbf{X}) - \xi_i^q] \\
  = \hspace{3pt}   & f(\mathbf{X}) + \sum_{q=1}^k \sum_{i\in Y_q\backslash X_q} \rho_{q,i}(\mathbf{X}) - \sum_{q=1}^k \sum_{i\in X_q \cup Y_q} \xi_i^q  \label{eq:rewrite2} \\
  \geq \hspace{3pt}   & f(X_1\cup Y_1, \dots, X_k \cup Y_k) - \sum_{q=1}^k \sum_{i\in X_q \cup Y_q} \xi_i^q \label{eq:k_sub_property}\\
  \geq \hspace{3pt}   & f(\mathbf{Y}) + \sum_{q=1}^k \sum_{i\in X_q\backslash Y_q} \rho_{q,i}(T_1, \dots, T_q\backslash \{i\}, \dots, T_k) - \sum_{q=1}^k \sum_{i\in X_q \cup Y_q} \xi_i^q \label{eq:k_sub_property2}\\
  \geq \hspace{3pt}   & f(\mathbf{Y}) + \sum_{q=1}^k \sum_{i\in X_q\backslash Y_q} \xi_i^q - \sum_{q=1}^k \sum_{i\in X_q \cup Y_q} \xi_i^q \label{eq:xi_property}\\
  = \hspace{3pt}   & f(\mathbf{Y}) - \sum_{q=1}^k \sum_{i\in Y_q} \xi_i^q \\
  = \hspace{3pt}   & f^*(\mathbf{Y}) = \hat f^*_{\mathbf{T}} (Y). \label{eq:f-star_def}
\end{align} 
\end{subequations}
\endgroup
Equations \eqref{eq:rewrite1}-\eqref{eq:rewrite2} rewrite $f^*$ in terms of $f$. Inequality \eqref{eq:k_sub_property} is a consequence of Corollary \ref{lemma:marginal} as we show next. For every $q\in\{1,\dots,k\}$, we fix an ordering of the elements in $Y_q\backslash X_q$ to be $(\alpha^q(1), \alpha^q(2), \dots, \alpha^q(|Y_q\backslash X_q|))$. Then 
\begin{align*}
 & f(X_1\cup Y_1, \dots, X_k \cup Y_k) \\
 = \hspace{3pt}   & f(\mathbf{X}) + \sum_{q=1}^k\sum_{j=1}^{|Y_q\backslash X_q|} \rho_{q, \alpha^q(j)}(X_1\cup Y_1,\dots, X_q\cup \{\alpha^q(r)\}_{r=1}^{j-1}, X_{q+1}, \dots, X_k)\\
 \leq \hspace{3pt}   & f(\mathbf{X}) + \sum_{q=1}^k\sum_{j=1}^{|Y_q\backslash X_q|} \rho_{q, \alpha^q(j)}(\mathbf{X})\\
 =  \hspace{3pt}   & f(\mathbf{X}) + \sum_{q=1}^k \sum_{i\in Y_q\backslash X_q} \rho_{q,i}(\mathbf{X}).
 \end{align*}  Similarly, inequality \eqref{eq:k_sub_property2} holds because 
\[ \rho_{q,i}(X_1\cup Y_1,\dots, X_q\cup Y_q \backslash\{i\}, \dots, X_k\cup Y_k) \geq \rho_{q,i}(T_1, \dots, T_q\backslash \{i\}, \dots, T_k) \] for any $i\in X_q\cup Y_q \subseteq T_q$ where $q\in\{1,\dots, k\}$. Inequality \eqref{eq:xi_property} follows from the definitions of $\xi_i^1$ and $\xi_i^2$. Equations \eqref{eq:f-star_def} follow from the definitions of $f^*$ and $\hat f^*$. 
\end{proof}

\begin{lemma}
\label{lemma:monotone_valid}
Let $f$ be a monotone $k$-submodular function. Given any $\mathbf{X}, \mathbf{S}\in (k+1)^N$, 
\begin{align*}
f(\mathbf{X}) \leq f(\mathbf{S})  & + \sum_{q=1}^k \sum_{i\in X_q\backslash \bigcup_{r=1}^k S_r} \rho_{q, i}(\mathbf{S})  +  \sum_{q=1}^k \sum_{p\in\{1,\dots, k\}\backslash \{q\}} \sum_{i\in X_q\cap S_p} \rho_{q,i}(\bfs{\emptyset}). 
\end{align*} 
\end{lemma}
\begin{proof}
Let \[X_q = \bigcup_{p=1}^kL^q_p \cup J_q, S_p = \bigcup_{q=1}^k L_p^q \cup K_p \] where $J_q, K_p, {L^q_p}$ are pairwise disjoint subsets of $N$ for all $p, q\in\{1,\dots,k\}$. Observe that $L_p^q = X_q \cap S_p$, $J_q = X_q \backslash \bigcup_{r=1}^k S_r$, and $K_p = S_p \backslash \bigcup_{r=1}^k X_r$, for all $p$ and $q$. Furthermore,
\begingroup
\allowdisplaybreaks
\begin{subequations}
\begin{align}
& f(\mathbf{S}) + \sum_{q=1}^k \sum_{i\in X_q\backslash \bigcup_{r=1}^k S_r} \rho_{q, i}(\mathbf{S})  +  \sum_{q=1}^k \sum_{p\in\{1,\dots, k\}\backslash \{q\}} \sum_{i\in X_q\cap S_p} \rho_{q,i}(\bfs{\emptyset}) \\
= \hspace{3pt} & f\left(\bigcup_{q=1}^k L_1^q \cup K_1, \dots, \bigcup_{q=1}^k L_k^q \cup K_k\right) + \sum_{q=1}^k \sum_{i\in J_q} \rho_{q, i}(\mathbf{S}) + \sum_{q=1}^k\sum_{p=1, p\neq q}^k \sum_{i\in L_p^q} \rho_{q,i}(\bfs{\emptyset})\\
\geq \hspace{3pt} & f\left(\bigcup_{q=1}^k L_1^q \cup K_1\cup J_1, \dots, \bigcup_{q=1}^k L_k^q \cup K_k\cup J_k\right) + f\left(\bigcup_{p=2}^k L^1_p, \dots, \bigcup_{p=1}^{k-1} L^k_p\right) \label{eq:c} \\
\geq \hspace{3pt} & f\left(L_1^1 \cup K_1\cup J_1, \dots, L_k^k \cup K_k\cup J_k\right) +  f\left(\bigcup_{p=2}^k L^1_p, \dots, \bigcup_{p=1}^{k-1} L^k_p\right) \label{eq:d}  \\
\geq \hspace{3pt} &  f\left(\bigcup_{p=1}^k L^1_p \cup K_1\cup J_1, \dots, \bigcup_{p=1}^{k} L^k_p \cup K_k\cup J_k\right) \label{eq:e}  \\
\geq \hspace{3pt} &  f\left(\bigcup_{p=1}^k L^1_p\cup J_1, \dots, \bigcup_{p=1}^{k} L^k_p\cup J_k\right) \label{eq:f} \\
= \hspace{3pt} & f(\mathbf{X}).
\end{align}
\end{subequations}
\endgroup 
Inequality \eqref{eq:c} follows from Corollary \ref{lemma:marginal}. Inequalities \eqref{eq:d} and \eqref{eq:f} are due to the monotonicity of $f$, and inequality \eqref{eq:e} holds because $f$ is $k$-submodular. 
\end{proof}

Lemma \ref{lemma:monotone_valid} applies to all monotone $k$-submodular functions. By using the relationship between any general $k$-submodular function $f$ and its monotone counterpart $f^*$ as stated in Lemma \ref{lemma:convert_to_monotone}, we obtain the following result. 

\begin{corollary}
\label{coro:not_monotone_valid}
Let $f$ be a $k$-submodular function. Given any $\mathbf{X}, \mathbf{S}\in (k+1)^N$, 
\begin{align*}
f(\mathbf{X}) \leq f(\mathbf{S})  & + \sum_{q=1}^k \sum_{i\in X_q\backslash \bigcup_{r=1}^k S_r} \rho_{q, i}(\mathbf{S})  +  \sum_{q=1}^k \sum_{p\in\{1,\dots, k\}\backslash \{q\}} \sum_{i\in X_q\cap S_p} \rho_{q,i}(\bfs{\emptyset}) - \sum_{q=1}^k \sum_{i\in S_q\backslash X_q} \xi_i^q.
\end{align*}
\end{corollary}

With these properties of $k$-submodular functions, we propose valid linear inequalities for the hypograph of any $k$-submodular function in the next section. 

\section{$k$-submodular Inequalities} 
\label{sect:k_sub_ineq}
Let $f$ be a $k$-submodular function defined over $N$. Recall that $\mathcal{T}_f$ is the epigraph of $f$. In this section, we propose two classes of valid linear inequalities for $\mathcal{T}_f$ depending on whether $f$ is monotone. We refer to these inequalities as the \emph{$k$-submodular inequalities}.

\begin{proposition}
Let $f$ be a monotone $k$-submodular function. For a given $\mathbf{S}\in (k+1)^N$, the inequality
\begin{equation}
\label{cut_monotone}
\begin{aligned}
\eta \leq f(\mathbf{S})  & + \sum_{q=1}^k \sum_{i\notin\bigcup_{r=1}^k S_r} \rho_{q, i}(\mathbf{S}) x_i^q  +  \sum_{q=1}^k \sum_{p\in\{1,\dots, k\}\backslash \{q\}} \sum_{i\in S_p} \rho_{q,i}(\bfs{\emptyset}) x_i^q
\end{aligned}
\end{equation}
is valid for $\mathcal{T}_f$.
\end{proposition}
\begin{proof}
Consider any $(\mathbf{x},\eta)\in\mathcal{T}_f$. Recall that $\mathbf{x} = {[{x^1}, \dots, {x^k}]}^\top \in \mathbb{B}^{kn}$. Let $\mathbf{X}\in (k+1)^N$ be the $k$-set represented by $\mathbf{x}$. For any $\mathbf{S}\in (k+1)^N$, 
\begingroup
\allowdisplaybreaks
\begin{subequations}
\begin{align}
\eta \leq & f(\mathbf{X}) \label{eq:cut_monotone_a}\\
\leq & f(\mathbf{S}) + \sum_{q=1}^k \sum_{i\in X_q\backslash \bigcup_{r=1}^k S_r} \rho_{q, i}(\mathbf{S})  +  \sum_{q=1}^k \sum_{p\in\{1,\dots, k\}\backslash \{q\}} \sum_{i\in X_q\cap S_p} \rho_{q,i}(\bfs{\emptyset}) \label{eq:cut_monotone_b}\\
= & f(\mathbf{S}) + \sum_{q=1}^k \sum_{i\notin\bigcup_{r=1}^k S_r} \rho_{q, i}(\mathbf{S}) x_i^q  +  \sum_{q=1}^k \sum_{p\in\{1,\dots, k\}\backslash \{q\}} \sum_{i\in S_p} \rho_{q,i}(\bfs{\emptyset}) x_i^q \label{eq:cut_monotone_c}
\end{align}
\end{subequations}
\endgroup 
Inequality \eqref{eq:cut_monotone_a} follows from the definition of $\mathcal{T}_f$. Inequality \eqref{eq:cut_monotone_b} holds by Lemma \ref{lemma:monotone_valid}. Equation \eqref{eq:cut_monotone_c} uses the characteristic vector $\mathbf{x}$ to equivalently state the set relations. To see this, for every $q\in\{1,\dots, k\}$ and $i\in N$, $x_i^q = 1$ exactly when $i\in X_q$.  
\end{proof}

\begin{proposition}
\label{prop:cut_valid}
Let $f$ be any $k$-submodular function. For a given $\mathbf{S}\in (k+1)^N$, the inequality 
\begin{equation}
\label{cut_general}
\begin{aligned}
\eta \leq f(\mathbf{S})  & + \sum_{q=1}^k \sum_{i\notin \bigcup_{r=1}^k S_r} \rho_{q, i}(\mathbf{S})x_i^q  +  \sum_{q=1}^k \sum_{p\in\{1,\dots, k\}\backslash \{q\}} \sum_{i\in S_p} \rho_{q,i}(\bfs{\emptyset})x_i^q - \sum_{q=1}^k \sum_{i\in S_q} \xi_i^q (1-x_i^q)
\end{aligned}
\end{equation}
is valid for $\mathcal{T}_f$.
\end{proposition}
\begin{proof}
Consider any $(\mathbf{x},\eta)\in\mathcal{T}_f$. Let $\mathbf{X}\in (k+1)^N$ be the $k$-set represented by $\mathbf{x}$. For any $\mathbf{S}\in (k+1)^N$, 
\begingroup
\allowdisplaybreaks
\begin{subequations}
\begin{align}
\eta \leq & f(\mathbf{X}) \label{eq:cut_general_a}\\
\leq & f(\mathbf{S}) + \sum_{q=1}^k \sum_{i\in X_q\backslash \bigcup_{r=1}^k S_r} \rho_{q, i}(\mathbf{S})  +  \sum_{q=1}^k \sum_{p\in\{1,\dots, k\}\backslash \{q\}} \sum_{i\in X_q\cap S_p} \rho_{q,i}(\bfs{\emptyset}) - \sum_{q=1}^k \sum_{i\in S_q\backslash X_q} \xi_i^q \label{eq:cut_general_b}\\
= & f(\mathbf{S}) + \sum_{q=1}^k \sum_{i\notin \bigcup_{r=1}^k S_r} \rho_{q, i}(\mathbf{S})x_i^q  +  \sum_{q=1}^k \sum_{p\in\{1,\dots, k\}\backslash \{q\}} \sum_{i\in S_p} \rho_{q,i}(\bfs{\emptyset})x_i^q - \sum_{q=1}^k \sum_{i\in S_q} \xi_i^q (1-x_i^q) \label{eq:cut_general_c}
\end{align}
\end{subequations}
\endgroup 
Inequality \eqref{eq:cut_general_a} holds due to the definition of $\mathcal{T}_f$, and \eqref{eq:cut_general_b} follows from Corollary \ref{coro:not_monotone_valid}. Lastly, for every $q\in\{1,\dots, k\}$ and $i\in N$, $x_i^q = 1$ or equivalently $1-x_i^q = 0$, exactly when $i\in X_q$. This justifies equation \eqref{eq:cut_general_c}.  
\end{proof}

We call inequalities \eqref{cut_monotone} and \eqref{cut_general} $k$-submodular inequalities associated with $\mathbf{S}\in(k+1)^N$. Intuitively, the first summation term on the right-hand side of a $k$-submodular inequality represents the marginal contribution made by appending additional elements to $S_q$, $q\in\{1,\dots,k\}$. The second nested summation term gives the upper bounds for the change in functional value when some elements in $S_q$ are switched to $S_{q'}$ for any $q'\neq q$. When $k=2$, we call the proposed inequalities \emph{bisubmodular inequalities}. In the next remark, we show that our proposed inequalities subsume the submodular inequalities.

\begin{remark}
Notice that the submodular inequalities proposed by \citet{nemhauser1981maximizing} is a special case of the $k$-submodular inequalities when $k=1$. Let $g:2^N\rightarrow \mathbb{R}$ be a submodular function defined on $N$. We denote the hypograph of $g$ by 
\[\{(y,\eta_g)\in \mathbb{B}^n \times \mathbb{R} \mid \eta_g \leq g(y) \}.\] 
For any $j\in N$ and $S\subseteq N$, we let $\gamma_j= g(N) - g(N\backslash \{j\})$ and $\rho_j(S) = g(S\cup\{j\}) - g(S)$. 
The submodular inequality associated with $S$ is 
\[ \eta_g \leq g(S) + \sum_{i\notin S} \rho_i(S) y_i - \sum_{j\in S} \gamma_j (1-y_j), \]
which is exactly inequality \eqref{cut_general} when $k=1$. In this submodular inequality, the first summation accounts for marginal returns from appending items to $S$, and the second summation estimates the change in function $g$ if items are removed from $S$. {The natural extension of a submodular inequality to the $k$-submodular setting is an inequality that accounts for the change in function value by adding an unselected item to, or removing an item from each of the $k$ subsets. However, the resulting inequalities are usually invalid because the function value also changes by switching a selected item from one subset to another. The $k$-submodular inequalities account for this  complication. } 
\end{remark} 

Now let us consider the polyhedron 
\begingroup
\allowdisplaybreaks
\begin{align*}
 \mathcal{P}_f = \{ (\mathbf{x},\eta)\in\mathbb{R}^{kn+1}  :  \eta \leq  f(\mathbf{S}) &+ \sum_{q=1}^k \sum_{i\notin \bigcup_{r=1}^k S_r} \rho_{q, i}(\mathbf{S})x_i^q  +  \sum_{q=1}^k \sum_{p\in\{1,\dots, k\}\backslash \{q\}} \sum_{i\in S_p} \rho_{q,i}(\bfs{\emptyset})x_i^q \\
 & - \sum_{q=1}^k \sum_{i\in S_q} \xi_i^q (1-x_i^q),  \forall \mathbf{S}\in (k+1)^N, \sum_{q=1}^k x^q_i \leq 1, \forall i\in N \}. 
\end{align*}
\endgroup

\begin{theorem}
\label{alter_Tf_description}
Given any $k$-submodular (not necessarily monotone) function $f$ and any $(\mathbf{x}, \eta)\in \mathbb{B}^{kn} \times \mathbb{R}$, we have $(\mathbf{x}, \eta)\in \mathcal{P}_f$ if and only if $\eta \leq f(\mathbf{X})$, where $\mathbf{X}$ is the $k$-set represented by $\mathbf{x}$.  
\end{theorem}
\begin{proof}
Suppose $(\mathbf{x}, \eta)\in \mathcal{P}_f$. Due to the second set of constraints in $\mathcal{P}_f$ and the fact that $\mathbf{x}\in\mathbb{B}^{kn}$, $x_i^q = 1$ exactly when $i\in X_q$ for any $i\in N$ and $q\in\{1,\dots,k\}$. In addition, $x_i^p = 0$ for all $p\in\{1,\dots,k\}\backslash \{q\}$. Therefore, 
\begingroup
\allowdisplaybreaks
\begin{align*}
\eta & \leq  f(\mathbf{X}) + \sum_{q=1}^k \sum_{i\notin \bigcup_{r=1}^k X_r} \rho_{q, i}(\mathbf{X})\cdot 0  +  \sum_{q=1}^k \sum_{p\in\{1,\dots, k\}\backslash \{q\}} \sum_{i\in X_p} \rho_{q,i}(\bfs{\emptyset})\cdot 0  - \sum_{q=1}^k \sum_{i\in X_q} \xi_i^q (1-1) \\
& = f(\mathbf{X}). 
\end{align*}
\endgroup
Conversely, suppose $\eta \leq f(\mathbf{X})$. Let $\mathbf{x}$ be the characteristic vector of the $k$-set $\mathbf{X}$. The second set of constraints in $\mathcal{P}_f$ trivially holds at $(\mathbf{x},\eta)$. For any $\mathbf{S}\in (k+1)^N$, 
\begingroup
\allowdisplaybreaks
\begin{align*}
\eta & \leq  f(\mathbf{X}) \\
& \leq f(\mathbf{S}) + \sum_{q=1}^k \sum_{i\in X_q\backslash \bigcup_{r=1}^k S_r} \rho_{q, i}(\mathbf{S})  +  \sum_{q=1}^k \sum_{p\in\{1,\dots, k\}\backslash \{q\}} \sum_{i\in X_q\cap S_p} \rho_{q,i}(\bfs{\emptyset}) - \sum_{q=1}^k \sum_{i\in S_q\backslash X_q} \xi_i^q \\
& = f(\mathbf{S}) + \sum_{q=1}^k \sum_{i\notin \bigcup_{r=1}^k S_r} \rho_{q, i}(\mathbf{S})x_i^q  +  \sum_{q=1}^k \sum_{p\in\{1,\dots, k\}\backslash \{q\}} \sum_{i\in S_p} \rho_{q,i}(\bfs{\emptyset})x_i^q  - \sum_{q=1}^k \sum_{i\in S_q} \xi_i^q (1-x_i^q), 
\end{align*}
\endgroup 
which follows from the argument in the proof of Proposition \ref{prop:cut_valid}. Thus $(\mathbf{x},\eta)$ satisfies the first set of constraints in $\mathcal{P}_f$ as well. We conclude that $(\mathbf{x},\eta)\in\mathcal{P}_f$. 
\end{proof}

\begin{corollary}
\label{cor:hypog}
Problem \eqref{eq:hypograph} is equivalent to
\begin{equation*}
\max\{\eta : (\mathbf{x},\eta)\in  \mathcal{P}_f \cap \mathbb{B}^{kn} \times \mathbb{R}, x\in\mathcal{K} \}.
\end{equation*}
\end{corollary}
\begin{proof}
This result directly follows from Theorem \ref{alter_Tf_description}. 
\end{proof}

\begin{remark}
It may be difficult to compute $\xi^q_i$, where $i\in N$ and $q\in\{1,\dots,k\}$, for non-monotone $k$-submodular functions in practice. However, we do not require exact $\xi^q_i$ values in the construction of the linear valid inequalities in our exact method. Proposition \ref{prop:cut_valid} still holds if we replace $\xi^q_i$ by its lower bound. One lower bound is $\zeta = \underline{f} - \overline{f}$, where $\underline{f}$ and $\overline{f}$ are a lower and an upper bound of $f$, respectively. This estimate can be improved depending on the problem context. Similarly, we can replace the $\xi^q_i$ values in $\mathcal{P}_f$ by their lower bounds that are cheaper to obtain. With the same proof, Theorem  \ref{alter_Tf_description} and Corollary \ref{cor:hypog} hold for the modified $\mathcal{P}_f$.  
\end{remark}

\section{A Cutting Plane Algorithm for $k$-submodular Maximization}
\label{sect:cut_plane}
We incorporate our proposed $k$-submodular inequalities in a cutting plane algorithm to tackle constrained $k$-submodular maximization problems in the form of \eqref{eq:original_max}, or equivalently \eqref{eq:hypograph}. Following the results in Section \ref{sect:k_sub_ineq}, problem \eqref{eq:hypograph} can be rewritten as 
\begin{subequations}
\label{eq:general_DCG_problem}
\begin{alignat}{2}
\max   &   \quad \eta  \\
\text{s.t.}   &   \quad (\mathbf{x}, \eta)\in \mathcal{C}, \label{constr:c} \\
  &   \quad \mathbf{x} \in \mathcal{K}.  \label{constr:k}
\end{alignat}
\end{subequations} 
The polyhedral set $\mathcal{C}$ in constraint \eqref{constr:c} is defined by the $k$-submodular inequalities, which provide a piecewise linear representation of the objective function $f$. The set $\mathcal{K}$ in constraint \eqref{constr:k} contains the characteristic vectors $\mathbf{x}$ that are associated with the feasible $k$-sets in $\mathcal{X}$. By abusing notation, $\mathcal{K}$ here also embeds the binary restriction $\mathbf{x}\in\mathbb{B}^{kn}$ and the constraints $\sum_{q=1}^k x^q_i \leq 1$, for all $i\in N$.\\

We propose Algorithm \ref{alg:DCG} to solve problem \eqref{eq:general_DCG_problem}. In this algorithm, we start with a relaxed set $\mathcal{C}$ and repeat the following subroutine until the optimality gap is within the given tolerance $\epsilon$. We solve a relaxed version of \eqref{eq:general_DCG_problem} to obtain $\overline{\mathbf{x}}$ and $\overline{\eta}$ { using a branch-and-bound algorithm}. The current solution $\overline{\eta}$ is an upper bound for the optimal objective, and $f(\overline{\mathbf{x}})$ serves as a lower bound. Let $\overline{\mathbf{X}}$ be the $k$-set that corresponds to $\overline{\mathbf{x}}$. If $\overline{\eta}$ overestimates $f(\overline{\mathbf{x}})$, then we restrict $\mathcal{C}$ by adding the $k$-submodular inequality \eqref{cut_general} associated with $\overline{\mathbf{X}}$. We repeat the same procedure in the next iteration.  \vspace{0.3cm} 

\begin{algorithm}[htb]
 \caption{Delayed Constraint Generation}
\label{alg:DCG}
\SetAlgoLined
\textbf{Input} initial $\mathcal{C}$, $\text{LB} = -\infty$, $\text{UB} = \infty$\;
 \While{ $(\text{UB}-\text{LB})/\text{UB}>\epsilon$ }{
  Solve problem  \eqref{eq:general_DCG_problem} {by a branch-and-bound algorithm} to get $(\overline{\mathbf{x}}, \overline{\eta})$\;
  \If{UB $> \overline{\eta}$}{
  $\text{UB}\leftarrow\overline{\eta}$\;
  }
     compute $f(\overline{\mathbf{x}})$\;
   \If{$\overline{\eta}>f(\overline{\mathbf{x}})$}{
     Add a $k$-submodular inequality \eqref{cut_general} associated with $\overline{\mathbf{x}}$ to $\mathcal{C}$\;
     }
     \If{$\text{LB}<f(\overline{\mathbf{x}})$}{
     $\text{LB}\leftarrow f(\overline{\mathbf{x}})$\;
    Update the incumbent solution to $\overline{\mathbf{x}}$ \;
     }
  }
  \textbf{Output} $\overline{\eta}$, $\overline{\mathbf{x}}$. 
\end{algorithm}

\begin{corollary}
Algorithm \ref{alg:DCG} converges to an optimal solution of Problem \eqref{eq:general_DCG_problem} in finitely many iterations. 
\end{corollary}
\begin{proof}
This result follows from the fact that the number of feasible solutions is finite and from Theorem \ref{alter_Tf_description}. 
\end{proof}

\section{Numerical Study}
\label{sect:comp}
In this numerical study, we demonstrate the effectiveness of our proposed Delayed Constraint Generation (DCG) Algorithm \ref{alg:DCG} by solving constrained $k$-submodular maximization problems {with $k=2$ and 3}. Specifically, we  run computational experiments on the {multi-type sensor placement problem}, described in Section \ref{subsect:sensor}. We refer the readers to Example 5.1 in \citep{yu2020polyhedral} for a small numerical example. To summarize, let a set $N$ of $n$ potential sensor deployment locations, and $t$ pairs of measurements made by all types of sensors at each location be given. Our goal is to determine a {multi-type sensor placement plan $\bfs{S}\in (k+1)^N$, subject to cardinality constraints $|S_q|\leq B_q$ for $q\in\{1,2,\dots,k\}$}, such that the entropy is maximized. Since the entropy function is highly nonlinear, we cannot formulate the {multi-type sensor placement problem} as a compact mixed-integer linear program. Therefore, we compare our DCG algorithm against the exhaustive search (ES) method, which is the only available benchmark. \\

Using the DCG approach, we formulate the {multi-type sensor placement problem} as  
\begingroup
\allowdisplaybreaks
 \begin{subequations}
\begin{alignat}{2}
\max \hspace{0.2cm}   &   \eta  \\
\textrm{s.t.}\hspace{0.2cm}    &   (\bfs{x}, \eta) \in \mathcal{C}, \label{subeq:ent_linearapprox}\\
  &   \sum_{q=1}^k x^q_i \leq 1,   &    &  \quad \text{ for all } i \in N, \label{subeq:ent_disjoint}\\
  &  \sum_{i\in N} x^q_i \leq B_q,   &    &  \quad \text{ for all } q\in\{1,2,\dots,k\}, \label{subeq:ent_card} \\
  &   x^q_i \in \mathbb{B},   &    &  \quad  \text{ for all } i\in N, q\in\{1,2,\dots,k\}. \label{subeq:binary}
\end{alignat}
\end{subequations}
\endgroup
 
The variables $x^q$ and $\eta$ are consistent with the notation in \eqref{eq:general_DCG_problem}. Constraint \eqref{subeq:ent_linearapprox} gives the piecewise linear representation of the entropy function by exploiting its $k$-submodularity. The inequalities \eqref{subeq:ent_card} ensure that the cardinality requirements are satisfied.  \\

We create random problem instances using the Intel Berkeley research lab dataset \citep{bodik2004intel}. {This dataset includes the sensor readings of three environmental factors---light, temperature, and humidity---at 54 locations in the Intel Berkeley Research lab from February 28th to April 5th in 2004.} We discretize the temperature data into three equal-width bins. Both light and humidity data are discretized into two equal-width bins. For the set of experiments with $k=2$, we aim to find the best placement plan for light and temperature sensors. {When $k=3$, our goal is to determine the optimal placement plan for light, temperature and humidity sensors.} The experiments are executed on two threads of a Linux server with Intel Haswell E5-2680 processor at 2.5GHz and 128GB of RAM. Our algorithms are implemented in Python 3.6 and Gurobi Optimizer 7.5.1 with default settings and one-hour time limit for each instance.   \\

First, we explore how the changes in the number of deployable locations, $n$, affect the computational performance of the DCG algorithm in both sets of experiments with $k=2$ {and $3$}.  We randomly select $n\in \{20, 30, 40, 50\}$ out of the 54 locations in the dataset. At each of the $n$ locations, we randomly select $t\in\{50, 100, 150, 200\}$ tuples of light, temperature and humidity measurements for evaluating the entropy. We set $B_q = \floor{n/10}$ for $q\in\{1,\dots, k\}$, so that the cardinality bound for each type of sensors increases proportionally with $n$. The computational results are summarized in Table \ref{res:entropy} for $k=2$ {and Table \ref{res:entropy_k3} for $k=3$}. The first two columns in these tables list the numbers of deployable sensor locations and the numbers of observations at each location. Columns 3-5 present the relevant computational statistics, namely the running time in seconds, the number of $k$-submodular inequalities added, and the number of branch-and-bound nodes visited when solving the relaxed master problems. The end optimality gap is computed by (UB$-$LB)/UB, where UB and LB are the best upper and lower bounds on the objective respectively. The last column reports the runtime of ES. At the time limit, ES does not provide end gap information because it produces no lower bounds and has to essentially go through each feasible solution to prove optimality. \\

\begin{table}[htb]
\small
\begin{center}
\captionsetup{font=footnotesize}
\caption{Computational performance of DCG and exhaustive search in the coupled sensor placement problem. The statistics are averaged across 3 trials. The superscript $^{\ell}$ means that out of the three trials, $\ell$ instances reach the time limit of one hour.} \label{res:entropy}
\begin{tabular}{c|c||c|c|c||c}
\hline
  $n$  &  $t$  &     time (s)   &     \# cuts  &     \# nodes &  ES time (s)  \\ 
  \hline
\multirow{4}{*}{20}	&	50	&	0.73	&	49.67	&	53.00	&	9.01 \\
	&	100	&	1.09	&	38.67	&	42.33	&	17.68 \\
	&	150	&	1.76	&	51.67	&	54.00	&	21.89 \\
	&	200	&	0.97	&	16.33	&	20.00	&	34.07 \\
\hline
\multirow{4}{*}{30}	&	50	&	7.67	&	285.33	&	289.67	&	--$^3$ \\
	&	100	&	13.04	&	250.33	&	253.67	&	--$^3$ \\
	&	150	&	15.07	&	224.67	&	230.00	&	--$^3$ \\
	&	200	&	23.70	&	241.67	&	245.00	&	--$^3$ \\
\hline
\multirow{4}{*}{40}	&	50	&	36.05	&	810.33	&	814.67	&	--$^3$ \\
	&	100	&	161.80	&	1783.33	&	1791.00	&	--$^3$ \\
	&	150	&	145.67	&	1255.33	&	1261.00	&	--$^3$ \\
	&	200	&	283.81	&	1676.67	&	1685.00	&	--$^3$ \\
\hline
\multirow{4}{*}{50}	&	50	&	104.21	&	1549.00	&	1560.00	&	--$^3$ \\
	&	100	&	478.09	&	3559.33	&	3566.33	&	--$^3$ \\
	&	150	&	1372.18	&	6911.67	&	6920.67	&	--$^3$ \\
	&	200	&	2474.35	&	9268.67	&	9275.67	&	--$^3$ \\
  \hline
\end{tabular}
\end{center}
\end{table}

\begin{table}[htb]
\small
\begin{center}
\captionsetup{font=footnotesize}
\caption{Computational performance of DCG and exhaustive search in the sensor placement problem with three types of sensors. The statistics are averaged across 3 trials. The superscript $^{\ell}$ means that out of the three trials, $\ell$ instances reach the time limit of one hour.} \label{res:entropy_k3}
\begin{tabular}{c|c||c|c|c|c||c}
\hline
  $n$  &  $t$  &     time (s)   &     \# cuts  &     \# nodes &  end gap & ES time (s)  \\ 
  \hline
 \multirow{4}{*}{20}	&	50	&	0.90	&	33.33	&	35.33	&	--	&	1633.00 \\
	&	100	&	1.58	&	32.67	&	35.33	&	--	&	2953.18 \\
	&	150	&	1.29	&	18.00	&	21.33	&	--	&	--$^3$ \\
	&	200	&	3.68	&	36.67	&	40.00	&	--	&	--$^3$ \\
\hline
\multirow{4}{*}{30}	&	50	&	7.54	&	142.67	&	147.33	&	--	&	--$^3$ \\
	&	100	&	20.27	&	222.33	&	226.67	&	--	&	--$^3$ \\
	&	150	&	22.48	&	156.67	&	160.33	&	--	&	--$^3$ \\
	&	200	&	46.55	&	243.33	&	248.33	&	--	&	--$^3$ \\
\hline
\multirow{4}{*}{40}	&	50	&	82.10	&	947.33	&	952.67	&	--	&	--$^3$ \\
	&	100	&	164.97	&	1043.00	&	1049.33	&	--	&	--$^3$ \\
	&	150	&	395.54	&	1599.67	&	1606.33	&	--	&	--$^3$ \\
	&	200	&	640.38	&	1902.67	&	1912.00	&	--	&	--$^3$ \\
\hline
\multirow{4}{*}{50}	&	50	&	253.29	&	1878.33	&	1886.00	&	--	&	--$^3$ \\
	&	100	&	1779.49	&	6907.67	&	6920.00	&	--	&	--$^3$ \\
	&	150	&	2719.26	&	6640.00	&	6646.33	&	--	&	--$^3$ \\
	&	200	&	--$^3$	&	6496.00	&	6500.67	&	6.10\%	&	--$^3$ \\
\hline
\end{tabular}
\end{center}
\end{table}

In this set of experiments, DCG solves all the instances when $k=2$ {and solves all but one test case when $k=3$, within the one hour time limit. The test case that DCG fails to solve attains a small end gap of 6.1\%. Based on the runtime differences, the instances with $k=3$ are in general more challenging than those with $k=2$, when the other parameters are kept the same. Overall, the computational statistics for $k=2$ and $k=3$ display the same trend.} The runtime, the number of branch-and-bound nodes as well as the number of $k$-submodular inequalities added increase as $n$ increases.  Variations in $t$ for small $n$ values do not significantly impact the computational statistics. When $n=50$, all the statistics increase at a greater rate in response to increments in $t$ compared with the case of $n=20$. On the other hand, unsurprisingly, ES struggles for $n\geq 30$ when $k=2$ {and $3$. In the test cases with $n=20$ that ES solves, the computing time drastically increases as $k$ goes from 2 to 3, reflecting the exponential growth of the decision space.} For $n=20,30$ and $40$, all instances are solved by DCG under 11 minutes; while ES hits the time limit for $n\ge 30$. In fact, when $k=2$, $n=50$ and $t=100$, exhaustive search needs to enumerate $50!/(5!5!40!) \approx 2.59\times 10^{12}$ feasible bisets to find an exact optimal solution. We find that objective function evaluation alone takes $1.6\times 10^{-4}$ seconds on average when $t=100$ for each biset. Thus, the total function evaluation time is equivalent to 13.13 years. In contrast, our algorithm finds an optimal solution in 8 minutes.\\

Next, we explore the effects of the cardinality bounds $B_q$, $q\in \{1,\dots, k\}$, on the computational performance of DCG. We consider all the placeable sensor locations; that is, $n=54$. Again, at each location, we randomly select $t\in\{50, 100, 150, 200\}$ $k$-tuples of sensor readings. We set $B_q = B$ for $q\in \{1,\dots, k\}$, where $B$ is an integer between 1 and 5. The computational results are summarized in Table \ref{res:entropy_v2} for $k=2$ {and Table \ref{res:entropy_k3_v2} for $k=3$}. In either table, the first column shows the upper bounds on the number of each type of sensors. The second column lists the numbers of observations at each of the 54 locations for entropy evaluations. The next four columns are the relevant computational statistics, including the runtime in seconds, the number of $k$-submodular inequalities added, the number of branch-and-bound nodes visited and the end optimality gaps. The last column reports the running time of ES in seconds. \\

\begin{table}[htb]
\small
\begin{center}
\captionsetup{font=footnotesize}
\caption{Computational performance of DCG and exhaustive search in the coupled sensor placement problem. The statistics are averaged across 3 trials. The superscript $^{\ell}$ means that out of the three trials, $\ell$ instances reach the time limit of one hour. }
\begin{tabular}{c|c||c|c|c|c||c}
\hline
  $B$  &  $t$  &     time (s)   &     \# cuts  &     \# nodes  & end gap &  { ES time (s)}  \\ 
  \hline
\multirow{4}{*}{1} 	&	50	&	1.09	&	38.00	&	41.00	&	--	&	0.48 \\
	&	100	&	1.16	&	18.33	&	21.33	&	--	&	1.04 \\
	&	150	&	1.45	&	16.67	&	20.00	&	--	&	1.33 \\
	&	200	&	1.97	&	14.67	&	17.33	&	--	&	1.98 \\
\hline
\multirow{4}{*}{2} 	&	50	&	11.03	&	283.33	&	286.67	&	--	&	546.64 \\
	&	100	&	27.05	&	333.33	&	337.67	&	--	&	1158.15 \\
	&	150	&	22.14	&	209.00	&	212.33	&	--	&	1596.22 \\
	&	200	&	47.81	&	283.33	&	286.00	&	--	&	2307.06 \\
\hline
\multirow{4}{*}{3} 	&	50	&	31.90	&	669.33	&	674.33	&	--	&	--$^3$ \\
	&	100	&	92.53	&	939.00	&	943.33	&	--	&	--$^3$ \\
	&	150	&	149.39	&	997.00	&	1000.00	&	--	&	--$^3$ \\
	&	200	&	335.15	&	1535.67	&	1541.33	&	--	&	--$^3$ \\
\hline
\multirow{4}{*}{4} 	&	50	&	67.92	&	1106.67	&	1114.33	&	--	&	--$^3$ \\
	&	100	&	404.21	&	3422.33	&	3428.33	&	--	&	--$^3$ \\
	&	150	&	754.14	&	3934.00	&	3940.00	&	--	&	--$^3$ \\
	&	200	&	1308.30	&	5154.00	&	5159.33	&	--	&	--$^3$ \\
\hline
\multirow{4}{*}{5} 	&	50	&	149.37	&	1773.67	&	1782.33	&	--	&	--$^3$ \\
	&	100	&	636.18	&	4467.33	&	4473.00	&	--	&	--$^3$ \\
	&	150	&	2068.81	&	8792.00	&	8800.33	&	--	&	--$^3$ \\
	&	200	&	2941.30$^1$	&	11385.00	&	11394.33	&	2.78\%	&	--$^3$ \\
     \hline
\end{tabular}
\label{res:entropy_v2}
\end{center}
\end{table}

\begin{table}[htb]
\small
\begin{center}
\captionsetup{font=footnotesize}
\caption{Computational performance of DCG and exhaustive search in the sensor placement problem with three types of sensors. The statistics are averaged across 3 trials. The superscript $^{\ell}$ means that out of the three trials, $\ell$ instances reach the time limit of one hour. }
\begin{tabular}{c|c||c|c|c|c||c}
\hline
  $B$  &  $t$  &     time (s)   &     \# cuts  &     \# nodes  &  end gap & ES time (s) \\ 
  \hline
\multirow{4}{*}{1}	&	50	&	1.14	&	19.33	&	22.00	&	--	&	37.56 \\
	&	100	&	3.05	&	31.00	&	33.67	&	--	&	64.11 \\
	&	150	&	4.80	&	34.00	&	36.67	&	--	&	94.67 \\
	&	200	&	5.33	&	26.00	&	28.00	&	--	&	134.80 \\
\hline
\multirow{4}{*}{2}	&	50	&	22.63	&	267.33	&	271.00	&	--	&	--$^3$ \\
	&	100	&	41.45	&	294.33	&	297.00	&	--	&	--$^3$ \\
	&	150	&	50.65	&	246.00	&	248.67	&	--	&	--$^3$ \\
	&	200	&	88.52	&	308.67	&	311.67	&	--	&	--$^3$ \\
\hline
\multirow{4}{*}{3}	&	50	&	103.74	&	923.00	&	927.67	&	--	&	--$^3$ \\
	&	100	&	184.81	&	1405.00	&	1410.67	&	--	&	--$^3$ \\
	&	150	&	355.86	&	1315.33	&	1319.67	&	--	&	--$^3$ \\
	&	200	&	524.26	&	1420.67	&	1424.00	&	--	&	--$^3$ \\
\hline
\multirow{4}{*}{4}	&	50	&	215.59	&	1624.00	&	1630.67	&	--	&	--$^3$ \\
	&	100	&	787.15	&	3315.00	&	3320.67	&	--	&	--$^3$ \\
	&	150	&	1618.03	&	4851.67	&	4857.00	&	--	&	--$^3$ \\
	&	200	&	2467.97	&	5447.00	&	5453.00	&	--	&	--$^3$ \\
\hline
\multirow{4}{*}{5}	&	50	&	386.29	&	2330.33	&	2339.67	&	--	&	--$^3$ \\
	&	100	&	1449.97	&	5362.33	&	5372.67	&	--	&	--$^3$ \\
	&	150	&	2871.68$^1$	&	8067.00	&	8073.00	&	4.34\%	&	--$^3$ \\
	&	200	&	--$^3$	&	7037.33	&	7044.00	&	6.26\%	&	--$^3$ \\	
     \hline
\end{tabular}
\label{res:entropy_k3_v2}
\end{center}
\end{table}

In both Tables \ref{res:entropy_v2} {and  \ref{res:entropy_k3_v2}}, higher cardinality bounds $B$ make the multi-type sensor placement problem more challenging, with longer running time, more cuts  added, and more branch-and-bound nodes  visited. In particular, when $B \geq 4$, the computational statistics increase at a higher rate as $t$ increases than that with $B\leq 3$. These trends  are true for both $k=2$ {and $k=3$}, since the decision space consisting of all the plausible deployment plans grows rapidly as more sensors are allowed. If, in addition, the number of observations is high, then each entropy evaluation becomes expensive, resulting in a significant increase in the running time. When $k=2$, even in the most challenging instances where $B=5$ and $t=200$, DCG solves two test instances within one hour, and obtains a feasible solution within 3\% optimality in the third trial. {Similarly, when $k=3$, DCG attains small optimality gap under 6.26\% for the most challenging test case with $B=5$ and $t=200$.} When the cardinality bounds are below three, DCG solves all the instances within six minutes for $k=2$ {and nine minutes for $k=3$}. On the contrary, ES fails due to the time limit for all the instances with $B\geq 3$ and $k=2$. {ES struggles more when $k=3$ and fails as soon as $B$ exceeds 1.} \\

\section{Concluding Remarks} 
\label{sect:conclusion}
In this paper, we propose a polyhedral approach to solve constrained maximization problems with $k$-submodular objective functions. We propose valid linear inequalities, referred to as $k$-submodular inequalities, for the hypograph of any $k$-submodular function. This development leads us to construct the first exact method---a delayed constraint generation algorithm based on $k$-submodular inequalities---to solve general $k$-submodular maximization problems other than the trivially available exhaustive search method. Our numerical experiments on a highly nonlinear multi-type sensor placement problem show that the proposed delayed constraint generation algorithm is effective when handling challenging $k$-submodular maximization problems that are unsolvable by existing  methods. 

\section*{Acknowledgements}
We thank the editor and the reviewers for providing comments that improved this paper. This research is supported, in part, by NSF grant 2007814. This research is also supported in part through the computational resources and staff contributions provided for the Quest high performance computing facility at Northwestern University, which is jointly supported by the Office of the Provost, the Office for Research, and Northwestern University Information Technology. 
\bibliography{max_bib}{}
\bibliographystyle{apalike}

\end{document}